\documentclass{IOS-Book-Article}     %[seceqn,secfloat,secthm]

\usepackage{amsxtra, amsthm}
\usepackage{amsfonts}
\usepackage{amssymb}
\usepackage{amscd}

\usepackage{cite, hyperref}

\hypersetup{
  colorlinks   = true, %Colours links instead of ugly boxes
  urlcolor     = blue, %Colour for external hyperlinks
  linkcolor    = blue, %Colour of internal links
  citecolor   = red %Colour of citations
}

\usepackage[  msc-links]{amsrefs}

\newtheorem{thm}{Theorem}
\newtheorem{prop}{Proposition}
\newtheorem{lem}{Lemma}
\newtheorem{rem}{Remark}
\newtheorem{cor}{Corollary}
\newtheorem{exa}{Example}

\newtheorem{defn}{Definition}

\def\Z{\mathbb Z}

\def\C{\mathbb C}
\def\P{\mathbb P^1 (k)}

\def\Z{{\mathbb Z}}
\def\C{{\mathbb C}}
\def\L{\mathcal L}

\def\M{\mathcal M}
\def\X{\mathcal X}

\def\p{\mathfrak p}

\def\p{\mathfrak p}
\def\s{\mathfrak s}

\def\a{\alpha}

\def\t{\tau}
\def\<{\langle}
\def\>{\rangle}

\def\Z{\mathbb Z}

\def\C{\mathbb C}

\def\C{\mathcal C}

\def\M{\mathcal M}

\def\a{\alpha}

\def\p{\mathfrak p}

\def\Z{\mathbb Z}

\def\C{\mathbb C}
\def\L{\mathcal L}

\def\M{\mathcal M}

\def\s{\sigma}
\def\a{\alpha}
\def\p{\mathfrak p}

\def\<{\langle}
\def\>{\rangle}
\def\X{\mathcal X}

\def\L{\mathcal L}

\def\P{\mathbb P^1 (k)}

\def\Z{{\mathbb Z}}
\def\C{{\mathbb C}}
\def\M{{\mathcal M}}

\def\L{\mathcal L}

\def\Aut{\mbox{Aut }}

\def\t{\tau}

\def\a{{\alpha }}

\def\<{\langle}
\def\>{\rangle}

\def\s{\mathfrak s}

\def\p{\mathfrak p}

\def\P{\mathbb P}
\DeclareMathOperator\ord{ord}
\DeclareMathOperator\img{Img}

\DeclareMathOperator\re{Re}
\DeclareMathOperator\mult{mult}

\begin{document}
\begin{frontmatter}          % The preamble begins here.
%
%\pretitle{}
\title{Weierstrass points of superelliptic curves}

\runningtitle{Superelliptic curves}
%\subtitle{}

% Two or more authors:
\author[A]{\fnms{C.} \snm{Shor}} 
\author[B]{\fnms{T.} \snm{Shaska}}

\runningauthor{Shor/Shaska}
\address[A]{Department of Mathematics, \\  Western New England University,    \\   Springfield, MA,  USA; \\ E-mail: cshor@wne.edu }
\address[B]{Department of Mathematics, \\ Oakland University, Rochester, MI, USA; \\ E-mail: shaska@oakland.edu}

\begin{abstract}
In this lecture we give a brief introduction to Weierstrass points of curves and computational aspects of $q$-Weierstrass points on superelliptic curves. 
\end{abstract}

\begin{keyword}
hyperelliptic curves,   superelliptic curves,  Weierstrass points
\end{keyword}

\end{frontmatter}

%%%%%%%%%%% The article body starts:

\section*{Introduction}

These lectures are prepared as a contribution to the NATO Advanced Study Institute held in Ohrid, Macedonia, in August 2014.
The topic of the conference was on the arithmetic of superelliptic curves, and this lecture will focus on the Weierstrass points of such curves.
 Since the Weierstrass points are an important concept of the theory of Riemann surfaces and algebraic curves and serve as a prerequisite for studying the automorphisms of the curves we will give in this lecture a detailed account of holomorphic and meromorphic functions of Riemann surfaces, the proofs of the Weierstrass and Noether gap theorems, the Riemann-Hurwitz theorem, and the basic definitions and properties of higher-order Weierstrass points. 
%All of this material is covered in Part.~1. 

Weierstrass points of algebraic curves are defined as a consequence of the important theorem of Riemann-Roch in the theory of algebraic curves. As an immediate application, the set of Weierstrass points is an invariant of a curve which is useful in the study of the curve's automorphism group and the fixed points of automorphisms.

In Part~1 we cover some of the basic material on Riemann surfaces and algebraic curves and their Weierstrass points. We describe some facts on the fixed points of automorphisms and prove the Hurwitz theorem on the bound of the automorphism group.  All the material is well known. We will assume that the reader has some familiarity with basic definitions of the theory of algebraic curves, such as divisors, Riemann-Roch theorem, etc.  We will provide some of the proofs and for the other results we give precise references. 

In Part~2, we describe linear systems and inflection points with an eye toward the Wronskian form, which is useful in computing these inflection points.  Then, using a special linear system, we are able to define Weierstrass points.  We give the basic definitions of Weierstrass points using Riemann-Roch spaces as well as with spaces of holomorphic differentials.  We generalize this definition to discuss higher-order Weierstrass points (which we call $q$-Weierstrass points).  Properties of these points, along with proofs and references, are given.  We conclude this part with Hurwitz's theorem, which gives an upper bound for the number of automorphisms of a curve of genus $g$.  Weierstrass points feature prominently in the proof.  We also use Weierstrass points to prove some bounds on the number of fixed points of automorphisms.

In Part~3, we examine Weierstrass points in a few contexts.  First, we see some results pertaining to Weierstrass points on superelliptic curves, which can be thought of as generalizations of hyperelliptic curves.  After that, we investigate group actions on non-hyperelliptic curves of genus $g=3$.

The material of these lecture will be used throughout this book, especially in \cite{nato-4}, \cite{nato-5}, or in \cite{nato-15}. Most of the material of this lecture is assumed as prerequisite for the rest of this volume. 
For further details and some open problems on Weierstrass points of weight $q\geq 2$ the reader can check \cite{w-points}.

\smallskip

\noindent \textbf{Notation:} 
Throughout this paper $\X_g$ will denote a smooth, irreducible algebraic curve, defined over and algebraically closed field $k$ of characteristic zero or equivalently a closed, compact Riemann surface of genus $g\geq 2$.

%******************************************************
\bigskip
\noindent \textbf{Part 1: Riemann surfaces and their meromorphic functions} \\

In this section we briefly describe some of the basic results on the theory of curves and divisors. We skip some of the proofs, but precise references are provided for each result. Most of the material can be found on the following classical books \cite{fulton, fk}

We assume that the reader is familiar with basic complex analysis and the basic definitions of  Riemann surfaces. 
%*********************************************
%\chapter{Functions on Riemann surfaces, meromorphic functions}
\section{Holomorphic and meromorphic functions on Riemann surfaces}
Let $X$ be a Riemann surface, $p\in X$ and $f:W \to \C$, such that $W$ is a neighborhood of $p$.
We say that $f$ is \textbf{holomorphic at $p$ } if there exists a chart $\Phi: U \to V$ with $p\in U$ such that $f\circ \Phi^{-1}$ is holomorphic at $\Phi(p)$.
Then, we say that   $f$ is \textbf{holomorphic on the neighborhood $W$} if it is holomorphic at every point of $W$.  
The following lemma is straightforward.

\begin{lem} Let $X$ be a Riemann surface, $p\in X$, and $f$ a complex valued function defined in a neighborhood $U$ of $p$, say $f : U \to \C$. Then, the following hold true:

\begin{enumerate}
\item $f$ is holomorphic at $p$ if and only if  for any chart $\Phi:U \to V$, $p\in U$ we have $f\circ \Phi^{-1}$ is a holomorphic at $\Phi(p).$
\item $f$ is holomorphic at $w$  if and only if  there exists a set of charts  $\left\{\Phi_i: U_i \to V_i \right\}$, 
with $W\subseteq \bigcup_i U_i$ such that $f\circ \Phi_i^{-1}$ is holomorphic on $ \Phi_i(W\cap U_i)$.
\item If $f$ is holomorphic at $p$, then $f$ is holomorphic at a neighborhood at $p$.
\end{enumerate}
\end{lem}

\noindent Next we will define singularities for Riemann surfaces.   Recall that for a function $f : \C \to \C$ we have defined singularities as follows: 

i)  $f(z)$ has a \textbf{removable singularity} at $z_0$ of it is possible to assign a complex number such that $f(z)$ becomes analytic, or $f(z)$ is bounded around $z_0$.

ii)  $f: U\backslash\{a\} \to \C$, where $U$ is open. Then, $a$ is an \textbf{essential singularity} if it is not a pole or a removable singularity.
 
Let $X$ be a Riemann surface, $p\in X$, $U$ a neighborhood of $p$, and $f: U \to \C$ a complex valued function and holomorphic. The function   $f$ is defined to have a  \textbf{removable singularity} at $p$  if and only if  $f\circ \Phi^{-1}$ has a removable singularity at $\Phi(p)$.     $f$ has a \textbf{pole at $p$}  if and only if  $f\circ \Phi^{-1}$ has a pole at $\Phi(p)$ \bigg (i.e. $f(z)=\frac{g(z)}{(z-\a)^n)}\bigg)$.  $f$ has an \textbf{essential singularity at $p$ }  if and only if  $f\circ \Phi^{-1}$ has an essential singularity at $\Phi(p)$.
 
\begin{lem} $f$ has a removable singularity if and only if  for every chart $\Phi: U \to V$ such that $p\in U$, the function $f\circ \Phi^{-1}$ has a removable singularity.
\end{lem}

Summarizing we have that for any  holomorphic  function $f: U\backslash \{p\} \to V$,  the following statements hold:  
\begin{enumerate}
\item  If $|f(x)|$ is bounded in a neighborhood of $p$, then $f$ has removable singularity at $p$. Moreover, the limit $\lim_{x \to p}f(x)$ exists, and if $f(p):=\lim_{x \to p}f(x)$,  then $f$ is holomorphic at $p$.

\item  If $|f(x)| \to \infty$ as $x \to \infty$, then $f(x)$ has a pole at $p$.

\item  If $|f(x)|$ has no limit as $x \to \infty$, then $f(x)$ has an essential singularity at $p$.
\end{enumerate}

\begin{defn} A function $f$ on $X$ is \textbf{meromorphic  at a point} $p\in X$ if it is either holomorphic, has a removable singularity, or has a pole at $p$.
$f$ is \textbf{meromorphic on $X$} if it is meromorphic at every point of $X$.
\end{defn}
The following are elementary properties of meromorphic functions. 
\begin{exa} Let $f,g$ be meromorphic on $X$. Then, $f\pm g, f\cdot g, \frac{f}{g}$ are meromorphic on $X$ provided $g(x)\neq 0$ (i.e, is not identically zero).
\end{exa}

If $W\subset X$ is an open subset of the Riemann surface $X$ we denote the set of meromorphic functions of $W$ by 
\[ \M_X(W)= \{f:W \to \C \, | \, f \thinspace \text{is meromorphic}\}. \]
%
%\textbf{Laurent series}
%
Let $f : U\backslash \{p\} \to V$ be holomorphic. Let $z$ be the local coordinate on $X$ near $p$. Hence, $z=\Phi(x)$, which implies that $f\circ \Phi^{-1}$ is holomorphic near $f(p):= z_0$. Then, there exist a series expansion 
\[ f\bigg (\Phi^{-1}(z)\bigg) =\sum _n c_n (z-z_0)^n, \] 
which is called the \textbf{Laurent series} for $f$ about $p$ with respect to $\Phi$.  The Laurent series tells us about the nature of singularity at $p$.
\begin{enumerate}
\item $f$ has a removable singularity at $p$  if and only if the Laurent series has n\textbf{o negative terms}.
\item $f$ has a pole at $p$  if and only if the Laurent series has \textbf{finitely many negative terms}.
\item $f$ has essential singularity  if and only if the Laurent series has \textbf{infinitely many negative terms at $p$}.
\end{enumerate}
Let $f$ be a meromorphic function at $p$, and $z$ some local coordinate around $p$. Let the Laurent series be given by 
\[ f\bigg (\Phi^{-1}(z)\bigg) =\sum _n c_n (z-z_0)^n. \] 
\textbf{The order of $f$ at $p$,} denoted by $\ord_p(f)$, is 
\[\ord_p(f)= min\{n \, | \, c_n\neq 0\}.\]
\begin{lem} $\ord_p(f)$ is independent of the choice of the local coordinate $z$.
\end{lem}

\proof
Let $\Psi : U^{\prime} \to V^{\prime}$ be another chart such that $p\in U^{\prime}$ and   $w=\Psi(x)$ near $p$ be   the local coordinate. Denote $w_0=\Psi(p)$.
The transition function $T(w)=\Phi\circ \Psi^{-1}$ expresses $z$ as a holomorphic function of $w$. 
If $T$ is invertible at $w_0$,  then  $T^{\prime}(w_0)\neq 0$. Hence, we have 
\[ z= T(w)= z_0+ \sum_{n\geq1} a_n(w-w_0)^n,\] 
where $a_1\neq 0$.

Let $\ord_p f(z)=c_{n_0}$ be the order of $f$ computed via $z$. Then, 
\[ z-z_0 = \sum_{n\geq1} a_n(w-w_0)^n\] 
is the Laurent series of $f(w)$ at $p$, where $z-z_0$ is the lowest term is  $a_{n_0}(w-w_0)^{n_0}$, and for $\sum_{n\geq1} a_n(w-w_0)^n$ the lowest possible order in $(w-w_0)$  
is  $c_{n_0}a_1^{n_0}(w-w_0)^{n_0}$. But $c_{n_0}\neq 0$, and $a_1\neq 0$, which implies that $\ord_p f(w)=n_0$. Hence,  $\ord_p f(z)=\ord_p f(w)=n_0$.  
\endproof
\noindent The following statements are true for any function $f$  holomorphic at $p$. 
\begin{enumerate}
\item $f$ is holomorphic at $p$  if and only if  $\ord_p f\geq 0$. Moreover, $f(p)=0$  if and only if  $\ord_p f>0$.
\item $f$ has a pole at $p$  if and only if  $\ord_p f<0$.
\item $f$ has either a zero or a pole at $p$  if and only if  $\ord_p f=0$. \\
\end{enumerate}
\noindent We say $f$ has a zero (resp. pole) of order $n$ at $p$ if $\ord_p f=n\geq1$ (resp. $\ord_p f= -n <0$). One immediately has the following.
\begin{lem}
\begin{enumerate}
\item $\ord_p(f\cdot g)=\ord_p(f) + \ord_p(g).$
\item $\ord_p\bigg(\frac{f}{g}\bigg)=\ord_p(f) - \ord_p(g).$
\item $\ord_p\bigg(\frac{1}{f}\bigg)=- \ord_p(f).$
\item $\ord_p  \left(f\pm g \right)\geq min\{\ord_p (f), \ord_p (g)\}.$ \\
\end{enumerate}
\end{lem}
%***************************
%\subsection{Rational functions}
Any rational function on the Riemann sphere is meromorphic, since it has only zeros and poles (no essential singularities). Let $f(z)=\frac{p(z)}{q(z)}$ be a rational function. Then, 
$ f(z)= c\cdot \prod_i (z-\lambda_i)^{e_i}$, 
where $\lambda_i$ are distinct complex numbers, and $e_i\in \Z$.
Thus,  $\ord_{\lambda_i}(f)=e_i$,  $\ord_{\infty} (f)=\deg p- \deg q$,  and $ \ord_x f=0$ if $x\neq \infty, \lambda_1,\cdots, \lambda_r$.
Also,  $\sum_{x\in X} \ord_x f=0$,  
where $X$ is the Riemann sphere.

% (this fact comes also from complex analysis).

%\subsection{Examples of Meromorphic Functions}

%\textbf{Meromorphic functions of the Riemann Sphere}

\begin{thm} Any meromorphic function of  $\C_{\infty}$ is a rational function. In other words,  $\M \left(\C_{\infty} \right)= \C(z)$.
\end{thm}

\proof
Let $f$ be meromorphic function on $\C_{\infty}$. Recall that $\C_{\infty}$ is compact and $f$ has finitely many zeros and poles. Let 
$ \{\lambda_1, \cdots , \lambda_r\}$ be the set of zeros and poles in $\C$. Assume that $\ord_{\lambda_i}(f)=e_i$.   Consider the function  $r(z)= \prod_i  \left(z-\lambda_i \right)^{e_i}$.  
Then, $r(z)$ and $f(z)$ have the same zeros and poles on $\C$. Then, $g(z)=\frac{f}{r(z)}$ is meromorphic function on $\C_{\infty}$, since 
$ f\in \M(\C_{\infty})$ and  $r\in \M  \left(\C_{\infty} \right)$.  
Thus,  $g(z)$ has no zeros and no poles in $\C$. Hence, $g(z)$ is holomorphic on $\C$, and so $g(z)$ has Taylor series $g(z)=\sum_{n=0}^{\infty} c_n z^n$, which converges everywhere on $\C$.

Since $g(z)$ is meromorphic at $z=\infty$, then $g \left(\frac{1}{z} \right)=\sum_{n=0}^{\infty} c_n \bigg(\frac{1}{z}\bigg)^n$ 
and  $g(w)=\sum_{n=0}^{\infty} c_n w^{-n}$ for a coordinate $w$, which means that $g(w)$ is meromorphic at $w=0$. This fact implies that $g(w)$ is a polynomial. 

If $g$ is constant, then $\frac{f}{r}$ is constant, and so $f$ is rational.
If $g$ is not constant, then it has no zeros in $\C$, and this is a contradiction.
\endproof

\begin{cor} Let $f$ be any meromorphic function on $\C_\infty$. Then,  $\sum_p \ord_p f=0$. 
\end{cor}

\proof Every meromorphic function is rational. 
%For rational functions this is true.
\endproof

%%%%*************************************************************
\subsection{Meromorphic functions on the projective line}

The same approach as above can be followed here with the only difference that we have to   homogenize to homogenous polynomials. The main result is the following.

\begin{thm} Every meromorphic function on $\P^1$ is a ratio of homogenous polynomials of the same degree.
\end{thm}

\begin{cor} Let $f$ be any meromorphic function on $\P^1$. Then, $\sum_p \ord_p f=0$. 
\end{cor}

\noindent The meromorphic functions on a complex torus are more difficult to describe as we will see in the next few paragraphs. 
%**********************************************
\subsection{Meromorphic functions on a complex torus}
Fix $\t$   in the upper plane and  consider the lattice $L=\Z+\t\Z$. Let $X$ be the complex torus $X=\C\diagup L$.  
%We construct meromorphic functions in $X$ by taking $L$-periodic meromorphic functions on $\C$.
For $\t \in \C$,  such that   $\img \t>0$  we define the theta-function  as
\begin{equation} \label{theta}
\Theta(z)=\sum_{n=-\infty}^{\infty} e^{\pi i(n^2\t+2nz)}.
\end{equation}
Then, the following hold: \\

i)  $\Theta(z)$ converges absolutely and uniformly on $\C$. 

ii)  $\Theta(z)$ is an analytic function on $\C$. 

iii)  $\Theta(z+1)= \Theta(z)$,  for every $z\in \C$, so $\Theta$ is periodic.

iv)  $\Theta(z+\t)= e^{-\pi i(\t+2z)}\cdot \Theta(z)$, $\forall z\in \C$. \\

\noindent Thus, $\Theta (z)$ is an analytic function and the series in Eq.~\eqref{theta} is its Fourier series. See the chapter on theta functions \cite{nato-4}  in this volume for more details.

\begin{prop} Fix a positive integer $d$, and choose any two sets of $d$ complex numbers $\{x_i\}$ and $\{y_i\}$ such that $\sum x_i-\sum y_i$ is an integer. Then, 
\[ \re (z)=\frac{\prod_i \Theta^{(x_i)}(z)}{\prod_j \Theta^{(y_i)}(z)} \] is a meromorphic $L$-periodic function on $\C$, and a meromorphic function on $\C\diagup L$.
\end{prop}

\proof We give a sketch of the proof. 
First show that   $\Theta(z_0)=0$  if and only if  $\Theta(z_0+m+n\t)=0$, $\forall m,n \in \Z$.  
Then, if  $\Theta(z_0)=0$, then $z_0$ is a simple zero.
 Show that all zeros of $\Theta$ are at $\frac{1}{2}+\frac{\t}{2}+m+n\t$. 
 
Denote by $\Theta^x (z):= \Theta \left(z-\frac{1}{2}-\frac{\t}{2}-x  \right)$ and check that
\begin{itemize}
\item $\Theta^x(z+1)=\Theta^x(z)$.
\item $\Theta^x(z+\t)=-e^{-2\pi i(z-x)}\cdot \Theta^x(z)$.
\end{itemize}
Then, 
$ \re (z)=\frac{\prod_i \Theta^{(x_i)}(z)}{\prod_j \Theta^{(y_i)}(z)} $
is meromorphic on $\C$, and 
$\re (z+1)=\re(z)$,  and  $\re   (z+\t)  \neq \re (z)$.
Indeed, 
\[\re (z+\t)= (-1)^{m-n} \cdot e^{-2\pi i \left[(m-n)z+\sum_j y_j-\sum_i x_i \right]}\cdot \re(z).\]
Hence, we need to show that 
\[(-1)^{m-n} e^{-2\pi i[(m-n)z+\sum_j y_j-\sum_i x_i]}=1,\] 
for all $z\in C$.
If $m=n$, and $\sum_j y_j-\sum_i x_i\in \Z$, then $e^{-2\pi is}=1$. Then, $\re(z)$ is meromorphic on $\C\diagup L$. The rest of the details are left to the reader. 
\endproof

Indeed,   every meromorphic function on $\C/L$ is of this form. 

\begin{thm}
Any meromorphic function of a complex torus is given as a ration of translated theta-functions. 
\end{thm}

\proof See \cite[Prop. 4.13, pg. 50]{Miranda}. \qed

\begin{rem}
The above theorem, highlights the special role of theta functions and why they are so important on the theory of algebraic curves.  In Chapter~\cite{nato-4} the reader can find even a more historical view of the important role of theta functions in development of algebraic geometry. 
\end{rem}

\begin{cor} Let $f$ be any meromorphic function on a complex torus $\C/L$. Then,  $\sum_p \ord_p f=0$. 
\end{cor}

%   proof    page 42

%********************************
\subsection{Meromorphic Functions on Smooth Plane Curves}

Let $X$ be a smooth plane curve defined by $f(x,y)=0$, where $x,y$ are holomorphic functions. Any polynomial $g(x,y)$ is holomorphic. This means that any rational $r(x)=\frac{g(x,y)}{h(x,y)}$ is meromorphic as long as $h(x,y)\neq 0$ on $X$.\\

Clearly: if $f(x,y) \, | \, g(x,y)$, then $h(x,y)\equiv 0$ on $X$. Indeed this is only when $h(x,y)$ could vanish:

\begin{thm}[Hilbert's Nullstellensatz]
Let $h(x_1,\cdots, x_2)$ be a polynomial vanishing everywhere, an irreducible polynomial $f(x_1,\cdots,x_n)$ vanishes. Then,$f \,| \, h$.
Hence, $\frac{g}{h}$ is meromorphic on the affine plane curve $f=0$ if $f\nmid h$.
\end{thm}

\begin{rem} $X$ is defined by $ f(x, y)=0$.  Let $I=\langle f\rangle$, which means that $k[x]= k[x,y]/ \langle I\rangle$, and $k(x)= k(x,y)/  \langle I\rangle$ is the quotient field of $k[x]$.
\end{rem}

%**********************************************************************************************************************
\section{Holomorphic functions between Riemann surfaces}

%\textbf{Goal:} To create the category of compact Riemann surface. So we need maps between objects.

%\subsection{Holomorphic maps}

Let $X, Y$ be two Riemann surfaces. A map $f:X \to Y$ is holomorphic at $p\in X$  if and only if  there exist charts: 
$ \Phi_1 : U_1 \to V_1$  on $X$ 
and    $\Phi_2 : U_2 \to V_2$ on $Y$ 
such that $p\in U_1, F(p)\in U_2$, and 
\[ \Phi_2\circ F\circ \Phi_1^{-1} \] 
is holomorphic at $\Phi_1(p)$.  Then we have the following lemma. 
\begin{lem} Let $F: X \to Y$ be a map between Riemann surfaces.  Then, we have

i)  $F$ is holomorphic at $p$  if and only if  for any pair of charts $\Phi_i: U_i \to V_i$, $i=1,2$ such that $p\in U_1$, $F(p)\in U_2$ we have $\Phi_2\circ F\circ \Phi_1^{-1}$ is holomorphic at $\Phi_1(p)$.

ii)  $F$ is holomorphic on $W$  if and only if  there are two collection of charts  $\{\Phi_1^{(i)}:U_1^{(i)} \to V_1^{(i)}\}$  on $X$ such that $ W\subset \bigcup_i U_1^{(i)}$, and  
$ \{\Phi_2^{(j)}:U_2^{(j)} \to V_2^{(j)}\}$   on $Y$ with  $ F(W)\subset \bigcup_j U_2^{(j)}$, such that $\Phi_2^{(j)}\circ F\bigg(\Phi_1^{(i)^{-1}}\bigg) $ is holomorphic for every $i,j$ where it is defined.

\end{lem}
\begin{proof} See \cite[Chapter II]{Miranda}
\end{proof}

Let $F: X \to Y$ as above. Then, for every open set $W\subset Y$ we have 
\[ \mathcal O_Y  (W)=\{ \textbf{the ring of holomorphic functions on  } \, W \}.\]
 $F^{-1}(W)$ is open in $X$. Then, we have $\mathcal O_X\bigg( F^{-1}(W)\bigg)$.
There is the induced map 
\[
\begin{split}
F^{*}: \mathcal O_Y (W) & \to \mathcal O_X \bigg(F^{-1}(W) \bigg)\\
f & \to f\circ \bar{F} \\
\end{split}
\] 
where $\bar{F}: F^{-1}(W) \to W$.    For meromorphic functions we have 
\[ 
\begin{split}
F^{*} : \M_Y(W) & \to \M_X \bigg(F^{-1}(W)\bigg) \\
              g & \to g\circ F. \\
\end{split}
\]
Next, we define   isomorphisms between   Riemann surfaces which will lead to the definition of an automorphism. 

%**********************************
\subsection{Isomorphisms of Riemann surfaces}

\begin{defn} An isomorphism between two Riemann surfaces is a holomorphic map $F: X \to Y$ which is bijective, and whose inverse is holomorphic.
\end{defn}

An isomorphism $F: X \to X$ is called an \textbf{automorphism}.  The following lemma is an elementary exercise which shows that the Riemann sphere and the projective line are isomorphic as Riemann surfaces. 

\begin{lem} The Riemann sphere $\C_{\infty}$ is isomorphic to the projective line $\P^1$.
\end{lem}

\noindent The proof is elementary and we only sketch it below.   Define the map $\Phi$ as follows: 
\[ 
\begin{split}
\Phi: \P^{\prime} & \to \C_{\infty} \\
  [z,w] & \to \bigg( 2\re(z\bar{w}), 2\img g(z\bar{w}), \frac{|z|^2-|w|^2}{|z|^2+|w|^2}\bigg),\\
\end{split}
\]
and show that it is an isomorphism.

\begin{prop} Let $f: X \to Y$ be a non-constant map between Riemann surfaces. Then for every $y\in Y$, $f^{-1}(Y)$ is a finite non-empty set of $Y$.
\end{prop}

\proof Fix a local coordinate $z$ around $y\in Y$. Let $x\in f^{-1}(y)$. Fix a local coordinate $w$ around $x\in X$. Then $f$ in terms of local coordinates is $z=g(w)$ (by the Implicit Function Theorem). Also, $g(x)=0$ since $z=0$ at $y$, and $w=0$ at $x$.

So the zeros of a holomorphic function are discrete, and in some neighborhood of $x$, we have $x$ as the only preimage of $y$. Hence, $f^{-1}(y)$ is discrete. But discrete subspaces of compact spaces are finite. This implies that $f^{-1}(y)$ is finite.
\endproof

\begin{lem} Let $f:X \to Y$ be a non-constant map between compact Riemann surfaces. Then, for any two points  $  x, y\in Y$ the fibers $ f^{-1}(x)$ and $ f^{-1}(y)$ have the same cardinality. 
\end{lem}

\proof  Exercise
\endproof

%\textbf{Riemann-Sphere}\\
%Let $f:X \to Y$ be a meromorphic map on the Riemann surface $X$. Define

%\begin{equation*}
%F(x):\left\{
%\begin{array}{rl}
%f(x)\in \C \thinspace \text{if x is not a pole} \\
%\infty \thinspace \text{if x is a pole}.
%\end{array}\right.
%\end{equation*}

%$F$ is holomorphic. There is a $1-1$ correspondence:
% \[ \{ \text{meromorphic functions} f on X \} \longrightarrow \{ \text{holomorphic maps} F: X \to \C_{\infty} \text{which are not ident infinity} \} \]

\subsection{Global Properties of Holomorphic Maps}

Next, we see some of the local and global properties of the holomorphic maps. 

\begin{prop}[Local Normal Form] 
Let $F: X \to Y$ be a non-constant holomorphic map defined at $p\in X$. Then, there is a unique integer $m\geq1$ such that: for every chart 
$  \Phi_2: U_2 \to V_2$
on $Y$   such that  $F(p)\in U_2$, there exists a chart  
$\Phi_1 : U_1 \to V_1$
on $X$ centered at $ p$ 
such that
\[ \Phi_2\bigg(F\big(\Phi_1^{-1}(z)\big)\bigg)=z^m.\]
\end{prop}

\proof See \cite[Prop. 4.1, pg. 44]{Miranda}
%By Taylor series \[ T(w)=\Phi_2 \bigg(F\big(\Phi^{-1}(w)\big)\bigg)=\sum_{i=m}^{\infty} c_{i} w^i \], such that $m\geq1$, since $T(0)=0$. (Take $z=0,w=0$)\\
%So , $T(w)=w^m S(w)$, where $S(w)$ is holomorphic and $S(0)\neq )$, which implies that $\exists R(w)$ holomorphic near zero, such that $S(w)=R(w)^m$.  Hence, $T(w)=\big( w R(w)\big)^{m}$.\\
%Let $\eta(w)= w R(w), \eta^{\prime}(0)=0$. Then, $\eta$ is invertible, and holomorphic (by Implicit function Thm). So $\Phi_1=\eta\circ \Psi$.\\
%Recall $ \Psi: U \to V$ such that $p\in U$.\\
%Complete....\\
\endproof

\begin{defn} 
The \textbf{multiplicity of $F$ at $p$}, denoted by $\mult_p F$ is the unique integer $m$ such that there are local coordinate near $p$ and $F(p)$ having the form $z\to z^m$.
\end{defn}

Notice that from the definition we have that for any $F : X \to Y$, the multiplicity of $p$ is  $\mult_p F \geq 1$. Take a local coordinate $z$ near $p$ and $w$ near $F(p)$ (i.e., $p$ corresponds to $z_0$ and $F(p)$ to $w$). Then, the map $F$ can be written as $w=h(z)$, where $h$ is holomorphic. Then, we have the following result. 

\begin{lem} The multiplicity $\mult_p F$ of $F$ at $p$ is 1 plus the order of the vanishing derivative $h^{\prime}(z_0)$ of $h$ at $z_0$. 
%=$\{order of the vanishing derivative $h^{\prime}(z_0)$ of $h$ at $z_0$, where $w=h(z)+1$ is holomorphic. 
In other words,  
\[ \mult_p F= 1+ \ord_{z_0} \left( \frac{d h}{dz} \right) \]
\end{lem}

\proof See \cite[Lemma 4.4]{Miranda} \endproof

\noindent If $h(z)$ is given as a power series around $z_0$ as 
\[ h(z) = h(z_0) + \sum_{i=m}^\infty c_i (z-z_0)^i, \]
with $m \geq 1$ and $c_m\neq 0$, then $\mult_p F = m$.

Since the points where the multiplicity $m \geq 2$ correspond to the zeroes of a holomorphic function, then there are finitely many of them. Hence, the following definition. 

 Let $F: X \to Y$ be a constant holomorphic map. A point $p\in X$ is called a \textbf{ramification point} for $F$ if $\mult_p F\geq 2.$
 A point $y\in Y$ is called a \textbf{branch point} for $F$ if it is the image of a ramification point for $F$.
 
The following lemma brings the above results in terms of the algebraic curves. We will skip the proof.

\begin{lem} 
i) Let $X$ be a smooth affine curve defined by $f(x,y)=0$. Define 
\[ 
\begin{split}
\pi_x:  X & \to \C \\
  (x,y) & \to x. \\
\end{split}
\]
Then,  $\pi_x$ is ramified at $p\in X$  if and only if  $\frac{\partial f}{\partial y}  (p) =0$.

ii) Let $X$ be a smooth projective plane curve defined by a homogenous polynomial $F(x, y, z)=0$ and $\phi: X \to \P^1$ the map  $[x, : y :z] \to [x, z]$.
Then, $\phi$ is ramified at $\p \in X$ if and only if $\frac {\partial F} {\partial y} (p) = 0$. 
\end{lem}

%***********************************************************
%\subsection{Degree of a holomorphic map between Riemann surfaces}

Next we will see the concept of the degree of a map between Riemann surfaces which will prepare us for the Riemann-Hurwitz formula. 

\begin{defn}
Let $ F: X \to Y $ be a non-constant holomorphic map between compact Riemann surfaces. For each $ y \in Y $, define 
\[ \deg_y (F)=\sum_{p\in F^{-1}(y)}\mult_p F. \] 
\end{defn}

\begin{prop}
Then, $ \deg_y (F) $ is constant independently of $y$. 
\end{prop}

\proof
The idea of the proof is to show that the function $y \to \deg_y (F)$ is locally constant. Since $Y$ is connected, then every locally constant function is constant.  We skip the details. \endproof

The above result motivates the following definition. Let $F: X \to Y$ be a non-constant holomorphic map between compact Riemann surfaces. The \textbf{degree of $F$} is defined as 
\[ \deg (F) = \deg_y (F),\]
for any $y\in Y$. 

\begin{cor} A holomorphic map between compact Riemann surfaces is an isomorphism if and only if it has degree equal to 1.
\end{cor}

\proof Degree equal to 1 means that the map is injective. But any non-constant holomorphic map is surjective, so it is an isomorphism.
\endproof

Let $F : X \to Y$ be a non-constant holomorphic map between compact Riemann surfaces. If we delete the branch points in $Y$ then we obtain a $\deg F  \to 1$ map, which is a \textbf{covering} in the topological sense.   Because of this, the  initial map $F : X \to Y$ is called a \textbf{branched covering}. 

The following is true also for Riemann surfaces, as expected.

\begin{prop} Let $f$ be a meromorphic function on a compact Riemann surface $X$. Then, $\sum_p \ord_p f =0$.
\end{prop}
The proof can be found in \cite[Prop. 4.12]{Miranda}, among many other places. 
%**********************************************************************************************
%\subsection{Covers of the projective line}

\subsection{Triangulations and the Euler's number}

Let $S$ be a Riemann surface. A \textbf{triangulation} on $S$ is a decomposition of $S$ into closed subsets, each holomorphic to a triangle, such that any two triangle are either  disjoint,  meet only at a single vertex,  have only an edge in common. 
 
Let a triangulation be given with $v$ vertices, $e$ edges, and $t$ triangles. 
The Euler number is defined as $e(S)= v -e + t$. 

The main result of the Euler number is that it does not depend on the particular triangularization.  Moreover, for a given genus $g$ surface, we can explicitly determine what this number is. We skip the details of the proof which can be found in most undergraduate texts on complex analysis and Riemann surfaces. 

\begin{prop}\label{euler} 
For a compact orientable Riemann surface of genus $g$, the Euler number is $2-2g$.
\end{prop}

%**************************************************************
%\subsection{Hurwitz formula}

Next we are ready to state and prove the Riemann-Hurwitz formula. 

\begin{thm}[Riemann-Hurwitz]
Let $F: X \to Y$ be a non-constant holomorphic map between compact Riemann surfaces, where the genus of $X$ 
(resp. the genus of $Y$) is $g(X)=g_X$ (resp.   $g(Y) = g_Y$). Then, 
\[ 2 (g_X-1)= 2 \deg F \left(g_Y-1 \right)+\sum_{p\in X} \left(\mult_p F-1\right). \]
\end{thm}

\proof First, it is worth noticing that since $X$ is compact, there is a finite set of ramification points. Therefore, $\sum_{p\in X}  \left(\mult_p F-1\right)$ is a finite sum. 

Take a triangulation on $Y$ such that each branch point is a vertex. Assume that there are    $v$ vertices, $e$ edges, and $t$ triangles. 
Every triangle in $Y$ will lift to a triangle in $X$. Let  $v^\prime$ vertices, $e^\prime$ edges, and $t^\prime$ triangles be the corresponding triangulation in $X$. 

Every ramification point is a vertex in $X$. A triangle lifts to $\deg F$ triangles in $X$. So $t^{\prime} = \deg F\cdot t$. Also, $e^{\prime} = \deg F\cdot e$.

Next we determine the number of vertices. 
Let $B$ be the set of branch points in $Y$ and $B^\prime$ the set of ramification points in $X$. Let  $q \in B$, so $q$ is a vertex in $Y$ and 
\[ \left| F^{-1}(q) \right|= \sum_{q\in F^{-1}(q)}1= \deg F- \sum_{q\in F^{-1}(q)} \left(\mult_p F-1\right). \]
Hence,
\[ 
\begin{split}
v^{\prime} & =\sum_{y\in B}   \left( \deg F- \sum \left(\mult_p F-1 \right)\right)= v \cdot \deg F -\sum_{ q\in B} \, \, \sum_{p\in F^{-1}(q)} \left(\mult_p F -1\right) \\
& = v\cdot deg F -\sum_{p\in B^\prime} \left(\mult_p F-1\right).  \\
\end{split}
\]
From Prop.~\ref{euler} we have that   $ 2-2g_X = v^\prime - e^\prime + t^\prime$. 
Hence,
\[
\begin{split}
2g_X -2  &= -v^{\prime}+e^{\prime}-t^{\prime}    \\
& = - v  \cdot \deg F + \sum_{p\in B^\prime} \left(\mult_p F-1\right)  + e\cdot \deg F-t\cdot \deg F \\
& = \deg F \cdot (-v+e-t) - \sum_{p\in B^\prime} \left(\mult_p F-1\right) \\
& = \deg F  \left(2g_Y -2 \right)+\sum_{p\in B^\prime} \left(\mult_p F-1\right). 
\end{split}
\]
Therefore,
\[2(g_X-1)=\deg F  \cdot \left(2g_Y-2\right) + \sum_{p\in X} \left(\mult_p F-1\right). \]
\endproof
The above theorem is one of the most used formulas in the area of Riemann surfaces and will be used repeatedly throughout this volume. 

%***************************************
\section{Riemann-Roch theorem}
Let $\X_g$ be a non-singular curve of genus $g$ defined over the field of complex numbers $k=\C$.  Let $k(\X_g)$ be the corresponding function field.  A \textbf{divisor} $D$ on $\X_g$ is a finite sum of points 
$D=\sum_{P\in C} n_P P$, for $n_P\in\mathbb{Z}$.  The set of all divisors is denoted $Div(\X_g)$. The degree of a divisor is the sum of the coefficients; that is, $\deg(D)=\sum_{P\in \X_g}n_P$.   Let the valuation of $D$ at $P$ be given by $\ord_P(D)=n_P$.  If $\ord_P(D)\geq0$ for all $P$, then we say $D$ is an effective divisor and write $D\geq0$.

For any $f\in k(\X_g)^\times$, let $(f)$ denote the (principal) divisor associated to $f$, and let $(f)_0$ and $(f)_\infty$ denote, respectively, the zero and pole divisors of $f$ so that $(f)=(f)_0-(f)_\infty$.  The valuation of $f$ at $P$, which is really the valuation of the principal divisor associated to $f$ at $P$, is denoted $\ord_P(f)$.  Two divisors $D_1$ and $D_2$ are said to be in the same divisor class if $D_1=D_2+(f)$ for some $f\in k(\X_g)^\times$.

%\begin{defn} A point $P$ is said to be an \textbf{inflection point} of $\X_g$ if there is a line given by the equation $f=0$ such that $\mult_P(f)\geq3$.
%\end{defn}

If $\omega\neq0$ is a meromorphic differential, then we define the divisor associated to $\omega$ analogously as $(\omega)=(\omega)_0-(\omega)_\infty$.  If $\omega_1$ and $\omega_2$ are two non-identically zero differentials, then $\omega_2/\omega_1$ is in $k(\X_g)$, and so the divisors associated to $\omega_1$ and $\omega_2$ are in the same class, which we call the \textbf{canonical class}.  The divisor associated to a differential is called a \textbf{canonical divisor}.  %A proof of the following can be found in \cite[III.4.9]{fk}.

%\begin{lem}
%Let $K$ be a canonical divisor of $\X_g$. Then,  $\deg (K) = 2g-2$. 
%
%\end{lem}

For any divisor $D$, the Riemann-Roch space is 
\[\L(D)=\left\{f\in k(\X_g)^\times : (f)+D\geq0\right\}\cup\left\{0\right\}.\]  
Let $\ell(D)$ denote the dimension of the vector space $\L(D)$.  Since the degree of any principal divisor is 0, $\L(D)=\emptyset$ if $\deg(D)<0$, and $\L(0)=k$.  The Riemann-Roch Theorem states that:

\begin{thm}[Riemann-Roch]
For any divisor $D$ and canonical divisor $K$, one has
\[ \ell(D) - \ell(K-D) = \deg(D)+1-g.\]
\end{thm}
\noindent For a proof, see \cite[III.4.8]{fk}.  In particular, 
\[ \ell(D)\geq\deg(D)+1-g.\]
A divisor $D$ for which $\ell(D)>\deg(D)+1-g$ is called \textbf{special}.

A few properties of canonical divisors follow immediately.  Using $D=0$ and $D=K$ with the above theorem, one finds that a canonical divisor $K$ has $\ell(K)=g$ and $\deg(K)=2g-2$.  This then implies that $\ell(D)=\deg(D)+1-g$ if $\deg(D)\geq 2g-1$.

\bigskip

%*************************************
\noindent \textbf{Part 2: Weierstrass points}

\smallskip

Next, we are ready to define inflection points and Weierstrass points and describe their properties. The material is classic and can be found in all classical books on the subject. Our favored reference is \cite{cornalba}.

\section{Linear systems, inflection points, and the Wronskian}\label{sec:linear-systems}
In this section, we describe inflection points of linear systems along with a method to calculate them which involves the Wronskian.  Our primary reference for this material is \cite{Miranda}.  A special linear system will lead us to Weierstrass points and higher-order Weierstrass points, which are described in more detail in the next section.

Let $D$ be a divisor on $\X_g$.  The \textbf{complete linear system} of $D$, denoted $|D|$, is the set of all effective divisors $E\geq0$ that are linearly equivalent to $D$; that is, \[ |D|=\{E\in Div(\X_g) : E=D+(f)\text{ for some } f\in \L(D)\}.\]  Note that any function $f\in k(\X_g)$ satisfying this definition will necessarily be in $\L(D)$ because $E\geq0$.  A complete linear system has a natural projective space structure which we denote $\P(\L(D))$.

We have previously seen $\L(D)$, the vector space associated to $D$.  Now, consider the projectivization $\P(\L(D))$ and the function \[S : \P(\L(D))\to|D|\] which takes the span of a function $f\in\L(D)$ and maps it to $D+(f)$.

\begin{lem}
If $X$ is a compact Riemann surface, then the map $S:\P(\L(D))\to|D|$ is a one-to-one correspondence.
\end{lem}
\proof
To show $S$ is surjective, suppose $E\in |D|$.  Then $E=D+(f)$ for some $f\in \L(D)$.  Thus, $S(f)=D+(f)=E$.

For injectivity, suppose $S(f)=S(g)$.  Then $(f)=(g)$, so $(f/g)=0$.  On a compact Riemann surface, the only functions without any zeroes or poles are constant functions.  Hence, $f/g$ is constant, so $f=\lambda g$ for some non-zero constant $\lambda$, which means $f$ and $g$ have the same span in $\L(D)$ and hence are equal in $\P(\L(D))$.
\qed

A \textbf{(general) linear system} is a subset of a complete linear system $|D|$ which corresponds to a linear subspace of $\P(\L(D))$.  The \textbf{dimension} of a general linear system is its dimension as a projective vector space.

Let $Q\subseteq |D|$ be a nonempty linear system on $\X_g$ with corresponding vector subspace $V\subseteq \L(D)$, and let $P\in\X_g$.  For any integer $n$, consider the vector space $V(-nP):=V \cap \L(D-nP),$ which consists of those functions in $\L(D)$ with order of vanishing at least $n$ at $P$.  This leads to a chain of nested subspaces \[V(-(n-1)P)\supseteq V(-nP)\] for all $n\in\Z$.  Since $\L(D-nP)=\{0\}$ for $n\geq\deg(D)$, this chain eventually gets to $\{0\}$.  As in Proposition~\ref{prop:dimension-increase-by-1}, which appears later, the dimension drops by at most 1 in each step.  We define gap numbers as follows.

\begin{defn}
An integer $n\geq1$ is a \textbf{gap number} for $Q$ at $P$ if $V(-nP)=V(-(n-1)P)-1$.  The set of gap numbers for $Q$ at $P$ is denoted $G_P(Q)$.
\end{defn}

Let $Q(-nP)$ denote the linear system corresponding to the vector space $V(-nP)$.  Then $Q(-nP)$ consists of divisors $D\in Q$ with $D\geq nP$.  An integer $n\geq1$ is a gap number for $Q$ at $P$ if and only if $\dim Q(-nP)=\dim Q(-(n-1)P)-1$.  A linear system $Q$ is called a $g_d^r$ if $\dim Q=r$ and $\deg Q=d$.  For such a system, the sequence of gap numbers is a $(r+1)$-element subset of $\{1,2,\dots,d+1\}$.  If this sequence is anything other than $\{1,2,\dots,r+1\}$, we call $P$ an \textbf{inflection point for the linear system $Q$}. 

Suppose the sequence of gap numbers is $\{n_1,n_2,\dots,n_{r+1}\}$, written in increasing order.  For each $n_i$, one can choose an element $f_i\in Q(-(n-1)P)\setminus Q(-nP)$.  Then $\ord_P(f_i)=n_i-1-\ord_P(D)$, and because of the different orders of vanishing at $P$, these functions are linearly independent, so $\{f_1,f_2,\dots,f_{r+1}\}$ is a basis for $V$.  Such a basis is called an \textbf{inflectionary basis} for $V$ with respect to $P$.  

Taken the other way, with a basis for $V$, through a change of coordinates, one can produce an inflectionary basis and hence construct the sequence of gap numbers.  Fix a local coordinate $z$ centered at $P$ and suppose $\{h_1,h_2,\dots,h_{r+1}\}$ is any basis for $V$.  Set $g_i=z^{\ord_P(D)}h_i$ for each $i$.  Then the functions $g_i$ are holomorphic at $P$ and thus have Taylor expansions \[g_i(z) = g_i(0) + g_i'(0)z + \frac{g_i^{(2)}(0)}{2!}z^2+\cdots+\frac{g_i^{(r)}(0)}{r!}z^r+\cdots.\]  We want to find linear combinations \[G_j(z)=\sum_{i=1}^{r+1} c_{i,j}g_i(z)\] of these functions to produce orders of vanishing from $0$ to $r$ at $P$.  This is possible precisely when the matrix
\[\left(\begin{matrix}
g_1(0) & g_1'(0) & g_1^{(2)}(0) & \cdots & g_1^{(r)}(0) \\ 
g_2(0) & g_2'(0) & g_2^{(2)}(0) & \cdots & g_2^{(r)}(0) \\ 
\vdots & \vdots & \vdots & \ddots & \vdots \\
g_{r+1}(0) & g_{r+1}'(0) & g_{r+1}^{(2)}(0) & \cdots & g_{r+1}^{(r)}(0)
\end{matrix}\right)\]
is invertible.  When that occurs, the same constants $c_{i,j}$ can be used to let $f_j=\sum_i c_{i,j}h_i$ and thus produce an inflectionary basis $\{f_j\}$ of $V$ such that $\ord_P(f_j)=j-1-\ord_P(D)$.  Thus, $G_P(Q)=\{1,2,\dots,r+1\}$ and so $P$ is an inflection point for $Q$.

\begin{defn}
The Wronskian of a set of functions $\{g_1,g_2,\dots,g_r\}$ of a variable $z$ is the function
\[W(g_1,g_2,\dots,g_r) = \det \left(\begin{matrix}
g_1(z) & g_1'(z) & g_1^{(2)}(z) & \cdots & g_1^{(r)}(z) \\ 
g_2(z) & g_2'(z) & g_2^{(2)}(z) & \cdots & g_2^{(r)}(z) \\ 
\vdots & \vdots & \vdots & \ddots & \vdots \\
g_{r+1}(z) & g_{r+1}'(z) & g_{r+1}^{(2)}(z) & \cdots & g_{r+1}^{(r)}(z)
\end{matrix}\right).\]
\end{defn}

As with its use in differential equations, the Wronskian is identically zero if and only if the functions $g_1,\dots,g_r$ are linearly dependent.

We summarize the work above with the following lemma.
\begin{lem}
Let $\X_g$ be a curve with a divisor $D$ and $Q$ a linear system corresponding to a subspace $V\subseteq\L(D)$.  Let $\{f_1,\dots,f_{r+1}\}$ be a basis for $V$, and for each $i$, let $g_i=z^{\ord_P(D)}f_i$.  Let $P$ be a point with local coordinate $z$.

Then $P$ is an inflection point for $Q$ if and only if $W(g_1,\dots,g_{r+1})=0$ at $P$.
\end{lem}

\begin{cor}
For a fixed linear system $Q$, there are finitely many inflection points.
\end{cor}
\proof See \cite[Lemma 4.4, Corollary 4.5]{Miranda}.\qed

Now, we consider higher-order differential forms.

\begin{defn}
A \textbf{meromorphic $n$-fold differential} in the coordinate $z$ on an open set $V\subseteq\mathbb{C}$ is an expression $\mu$ of the form $\mu=f(z)(dz)^n$ where $f$ is a meromorphic function on $V$.
\end{defn}

Suppose $\omega_1,\dots,\omega_m$ are meromorphic $1$-fold differentials in $z$ where $\omega_i=f_i(z)dz$ for each $i$.  Then their product is defined locally as the meromorphic $m$-form $f_1\cdots f_m (dz)^m$.  With this, we consider the Wronskian.

\begin{lem}
Let $\X_g$ be an algebraic curve with meromorphic functions $g_1,\dots,g_m$.  Then $W(g_1,\dots,g_m)(dz)^{m(m-1)/2}$ defines a meromorphic $m(m-1)/2$-fold differential on $\X_g$.
\end{lem}
\proof Since each $g_i$ is meromorphic, the Wronskian is as well, and so this is clearly a meromorphic $m(m-1)/2$-fold differential locally.  What remains to be shown is that the local functions transform to each other under changes of coordinates.  See \cite[Lemma 4.9]{Miranda} for the details.\qed

From here on, let $W(g_1,\dots,g_m)$ denote this meromorphic $m(m-1)/2$-fold differential.  We now look more closely at the poles of the Wronskian.

As with meromorphic functions and meromorphic $1$-forms, the order of vanishing of a meromorphic $n$-fold differential $f(z)(dz)^n$ is given by \[\ord_P(f(z)(dz)^n) = \ord_P(f(z)).\]  Divisors are defined in a similar way; namely, \[(\mu) = \sum_{P}\ord_P(\mu)P.\]

With these definitions, we can consider spaces of meromorphic $n$-fold differentials whose poles are bounded by $D$.  Namely, let \[\mathcal{L}^{(n)}(D)=\{\mu \text{ a meromorphic $n$-fold differential} : (\mu)\geq-D\}.\]  Equivalently, for a local coordinate $z$, if $(dz)=K$, then \[L^{(n)}(D)=\{f(z)(dz)^n : f\in\mathcal{L}(D+nK)\}.\]

\begin{lem}\label{lem:bounded-pole-orders}
Let $D$ be a divisor on an algebraic curve $\X_g$.  Let $f_1,\dots,f_m$ be meromorphic functions in $\L(D)$.  Then the meromorphic $n$-fold differential $W(f_1,\dots,f_m)$ has poles bounded by $mD$.  That is, \[W(f_1,\dots,f_m)\in\L^{m(m-1)/2}(mD).\]
\end{lem}
\proof Fix a point $P$ with local coordinate $z$.  For each $i$, let $g_i=z^{\ord_P(D)}f_i$ so that the $g_i$'s are holomorphic at $P$.  Then the Wronskian $W(g_1,\dots,g_m)$ is holomorphic at $P$ as well.  Since the Wronskian is multilinear, \[W(z^{\ord_P(D)}f_1,\dots,z^{\ord_P(D)}f_m)=z^{m\cdot\ord_P(D)}W(f_1,\dots,f_m).\]  Since this is holomorphic at $P$, we have $\ord_P(W(f_1,\dots,f_m))\geq -mD$ as desired.
\qed

Suppose $\{f_1,\dots,f_{r+1}\}$ and $\{h_1,\dots,h_{r+1}\}$ are two bases for a subspace $V\subseteq\L(D)$ with corresponding linear system $Q\subseteq|D|$.  Consider the Wronskian of each basis.  Since we have a change of basis matrix to transform from the $f_i$'s to the $h_j$'s, the Wronskian is scaled by the determinant of such a matrix which is a scalar and thus doesn't affect the zeroes or poles.  Therefore, the Wronskian is well-defined (up to a scalar multiple) by the linear system $Q$ rather than the choice of a basis.  We denote this Wronskian by $W(Q)$ and see that \[W(Q)\in\L^{(r(r+1)/2)}((r+1)D)\] by Lemma~\ref{lem:bounded-pole-orders}.

\begin{prop}
For an algebraic curve $\X_g$ of genus $g$ with linear system $Q$ of dimension $r$, \[\deg(W(Q))=r(r+1)(g-1).\]
\end{prop}
\proof The proof follows from the fact that $W(Q)$ is a meromorphic $r(r+1)/2$-fold differential of the form $f(z)(dz)^{r(r+1)/2}$ for some local coordinate $z$.  Since $f(z)$ is meromorphic, the degree of $(f(z))$ is zero.  And on a curve of genus $g$, the degree of $(dz)$ is $2g-2$.  Thus, the degree of $(f(z)(dz)^{r(r+1)/2})$ is \[\dfrac{r(r+1)}{2}(2g-2)=r(r+1)g-1.\]
\qed

We define the \textbf{inflectionary weight} of a point $P$ with respect to a linear system $Q$ to be \[w_P(Q)=\sum_{i=1}^{r+1}(n_i-i),\]  where $\{n_1,\dots,n_{r+1}\}$ is the sequence of gap numbers for $Q$ at $P$ written in ascending order.  It follows that $P$ is an inflection point for $Q$ precisely when $w_P(Q)>0.$  It turns out that the inflectionary weight of $P$ is exactly the order of vanishing of the Wronskian at $P$.

\begin{lem}
If $G_P(Q)=\{n_1,\dots,n_{r+1}\}$ and $\{f_1,\dots,f_{r+1}\}$ is a basis for $V$, then 
\[ w_P(Q)=\ord_P(W(z^{\ord_P(D)}f_1,\dots,z^{\ord_P(D)}f_{r+1})).\]
\end{lem}
\proof See \cite[Lemma 4.14]{Miranda}.
\qed

\begin{thm}\label{thm:total-inf-weight}
For $\X_g$ an algebraic curve of genus $g$ with $Q$ a $g_d^r$ on $\X_g$, the total inflectionary weight on $\X_g$ is \[\sum_{P\in\X_g}w_P(Q)=(r+1)(d+rg-r).\]
\end{thm}
\proof
Choose a basis $\{f_1,\dots,f_{r+1}\}$.  Then 
\begin{align*}
\sum_{P\in\X_g}w_P(Q) &= \sum_{P\in\X_g}  \\
& = \sum_{P\in\X_g} \ord_P(W(z^{\ord_P(D)}f_1,\dots,z^{\ord_P(D)}f_{r+1})) \\
& = \sum_{P\in\X_g} [(r+1)\ord_P(D) +\ord_P(W(Q))] \\ 
& = (r+1)d + r(r+1)(g-1).
\end{align*}
\qed

We now consider a special linear system, namely the canonical linear system, $Q=K$.  Inflection points for this system are called \textbf{Weierstrass points}, and the \textbf{Weierstrass weight} of such a point is its inflectionary weight with respect to $K$.

By Riemann-Roch, $\dim |K|=g-1$ and $\deg K=2g-2$.

\begin{cor}
The total Weierstrass weight on a curve of genus $g$ is $g^3-g=(g+1)g(g-1)$.
\end{cor}
\proof Theorem~\ref{thm:total-inf-weight} with $d=2g-2$ and $r=g-1$.\qed

For any $q\geq1$, we use the linear system $qK$ to define $q$-Weierstrass points, which have $q$-Weierstrass weights.  For $q=1$, the results are above.  For $q=2$, $d=\deg qK=q(2g-2)$ and $r=\dim |qK|=(2q-1)(g-1)$.
\begin{cor}\label{cor:total-q-weight}
The total $q$-Weierstrass weight, for $q\geq2$, on a curve of genus $g$ is $g(g-1)^2(2q-1)^2$.
\end{cor}

\begin{rem}
There are $q$-Weierstrass points for any curve of genus $g>1$ and any $q\geq 1$.
\end{rem}

\section{Introduction to Weierstrass points}

In this section, we use divisors on algebraic curves (following the notation of \cite{stichtenoth}) to give a more intuitive introduction to Weierstrass points on curves defined over $\C$.  For curves in positive characteristic, the situation is somewhat different; see \cite{Neeman, MR568202, voloch}.  We then introduce higher-order Weierstrass points, which we call $q$-Weierstrass points.  We conclude this section using results from the previous section to get a bound on the number of $q$-Weierstrass points, which will be useful later in computing an upper bound for the size of $\Aut(\X_g)$.

\subsection{Weierstrass points via gap numbers}

Let $P$ be a point on $\X_g$ and consider the vector spaces $\L(nP)$ for $n=0,1,\dots,2g-1$.  These vector spaces contains functions with poles only at $P$ up to a specific order.  This leads to a chain of inclusions \[ \L(0)\subseteq \L(P) \subseteq \L(2P) \subseteq \dots \subseteq \L((2g-1)P)\] with a corresponding non-decreasing sequence of dimensions \[ \ell(0) \leq \ell(P) \leq \ell(2P) \leq \dots \leq \ell((2g-1)P).\]
The following proposition shows that the dimension goes up by at most 1 in each step.
\begin{prop}\label{prop:dimension-increase-by-1} For any $n>0$, \[ \ell((n-1)P)\leq \ell(nP)\leq \ell((n-1)P)+1.\]\end{prop}
\proof
It suffices to show $\ell(nP)\leq \ell((n-1)P)+1$.  To do this, suppose $f_1, f_2\in\ell(nP)\setminus\ell((n-1)P).$  Since $f_1$ and $f_2$ have the same pole order at $P$, using the series expansions of $f_1$ and $f_2$ with a local coordinate, one can find a linear combination of $f_1$ and $f_2$ to eliminate their leading terms.  That is, there are constants $c_1, c_2\in k$ such that $c_1f_1+c_2f_2$ has a strictly smaller pole order at $P$, so $c_1f_1+c_2f_2\in\L((n-1)P)$.  Then $f_2$ is in the vector space generated by a basis of $\L((n-1)P)$ along with $f_1$.  Since this is true for any two functions $f_1,f_2$, we conclude $\ell(nP)\leq \ell((n-1)P)+1$, as desired.\qed

For any integer $n>0$, we call $n$ a \textbf{Weierstrass gap number of $P$} if $\ell(nP)=\ell((n-1)P)$; that is, if there is no function $f\in k(\X_g)^\times$ such that $(f)_\infty=nP$.  Weierstrass stated and proved the ``gap'' theorem, or \textbf{L\"uckensatz}, on gap numbers in the 19th century, likely in the 1860s.
\begin{thm}[The Weierstrass ``gap'' theorem]For any point $P$, there are exactly $g$ gap numbers $\alpha_i(P)$ with \[1=\alpha_1(P)<\alpha_2(P)<\cdots<\alpha_g(P)\leq2g-1.\]
\end{thm}
This theorem is a special case of the Noether ``gap'' theorem, which we state and prove below.

The set of gap numbers, denoted by $G_P$,  forms the \textbf{Weierstrass gap sequence} for $P$.  The non-gap numbers form a semi-group under addition since they correspond to pole orders of functions.

\begin{defn}[Weierstrass point]
If the gap sequence at $P$ is anything other than $\{1,2,\dots,g\}$, then $P$ is called a \textbf{Weierstrass point}.
\end{defn}
Equivalently, $P$ is a Weierstrass point if $\ell(gP)>1$; that is, if there is a function $f$ with $(f)_\infty=mP$ for some $m$ with $1<m\leq g$.

The notion of gaps can be generalized, which we briefly describe.  Let $P_1,P_2,\dots,$ be a sequence of (not necessarily distinct) points on $\X_g$.  Let $D_0=0$ and, for $n\geq1$, let $D_n=D_{n-1}+P_n$.  One constructs a similar sequence of vector spaces \[\L(D_0)\subseteq\L(D_1)\subseteq\L(D_2)\subseteq\dots\subseteq\L(D_n)\subseteq\cdots\] with a corresponding non-decreasing sequence of dimensions \[\ell(D_0)<\ell(D_1)<\ell(D_2)<\dots<\ell(D_n)<\cdots.\]  If $\ell(D_n)=\ell(D_{n-1})$, then $n$ is a \textbf{Noether gap number} of the sequence $P_1, P_2, \dots .$

\begin{thm}[The Noether ``gap" theorem] 
For any sequence $P_1,P_2, \dots$, there are exactly $g$ Noether gap numbers $n_i$ with \[1=n_1<n_2<\dots<n_g\leq 2g-1.\]
\end{thm}
\proof In analog with Proposition~\ref{prop:dimension-increase-by-1}, one can show the dimension goes up by at most 1 in each step; that is, \[ \ell(D_{n-1})\leq \ell(D_n)\leq \ell(D_{n-1})+1\] for all $n>0$.  First, note that the Riemann-Roch theorem is an equality for $n>2g-1$, so the dimension goes up by 1 in each step, so there are no gap numbers greater than $2g-1$.

Now, consider the chain $\L(D_0)\subseteq\dots\subseteq\L(D_{2g-1})$.  By Riemann-Roch, $\ell(D_0)=1$ and $\ell(D_{2g-1})=g$, so in this chain of vector spaces, the dimension must increase by 1 exactly $g-1$ times in $2g-1$ steps.  Thus, for $n\in\{1,2,\dots,2g-1\}$,  there are $g$ values of $n$ such that $\ell(D_n)=\ell(D_{n-1})$.  These $g$ values are the Noether gap numbers.
\qed

\begin{rem}The Weierstrass ``gap'' theorem is a special case of the Noether ``gap'' theorem, taking $P_i=P$ for all $i$.\end{rem}
%This result is a direct application of the Riemann-Roch theorem, and the proof can be found in \cite[III.5.4]{fk}.

Since a Weierstrass gap sequence contains $g$ natural numbers between $1$ and $2g-1$, and since its complement in $\mathbb{N}$ is a semi-group, we can begin to list the possible gap sequences for points on curves of small genus.%\footnote{This paragraph and what follows immediately below don't need to go here, nor do they need to be included at all.}
\begin{itemize}
\item For $g=1$, the only possible gap sequence is $\{1\}$.  Note that this means a curve of genus $g=1$ has no Weierstrass points.
\item For $g=2$, the possible sequences are $\{1,2\}$ and $\{1,3\}$.
\item For $g=3$, the possible sequences are $\{1,2,3\}, \{1,2,4\}, \{1,2,5\}, \{1,3,5\}$.
\end{itemize}
Two questions immediately arise.  First, given $g$, how many possible sequences are there?  Second, for each sequence, is there a curve $\X_g$ with a point $P$ that has that given sequence?

Regarding the first question, it has been shown that, for $N_g$ the number of sequences for genus $g$ has Fibonacci-like growth; namely, that \[\lim_{g\to\infty}N_g\phi^{-g}=S\] where $\phi=\frac{1+\sqrt{5}}{2}$ and $S$ is a constant.  This result, as well as references to other estimates on $N_g$, can be found in \cite{zhai}.

As to the second question, it has been shown that the answer in general is no.  In \cite{buchweitz}, Buchweitz gives an example of a sequence with $g=16$ for which there is no curve $\X_{16}$ with a point $P$ that has that sequence.  On the other hand, it has been shown that every sequence for $g\leq9$ is possible; see \cite{komeda}.

%***********************************************************
\subsection{Weierstrass points via holomorphic differentials}
Continuing with a point $P$ on a curve $\X_g$, recall that $n$ is a gap number precisely when $\ell(nP)=\ell((n-1)P)$.  By Riemann-Roch, this occurs exactly when \[\ell(K-(n-1)P)-\ell(K-nP)=1\] for a canonical divisor $K$, which is the divisor associated to some differential $dx$.  Thus there is $f\in k(\X_g)^\times$ such that $(f)+K-(n-1)P\geq0$ and $(f)+K-nP\not\geq0$, which implies that $\ord_P(f\cdot dx)=n-1$.  Since $(f)+K\geq(n-1)P\geq0$ (for $n\geq1$), $n$ is a gap number of $P$ exactly when there is a holomorphic differential $f\cdot dx$ such that $\ord_P(f\cdot dx)=n-1$.

For $H^0(\X_g,\Omega^1)$ the space of holomorphic differentials on $\X_g$, by Riemann-Roch, the dimension of $H^0(\X_g,\Omega^1)$ is $g$.  Let $\{\psi_i\}$, for $i=1,\dots, g$, be a basis, chosen in such a way that \[\ord_P(\psi_1)<\ord_P(\psi_2)<\cdots<\ord_P(\psi_g).\]  Let $n_i=\ord_P(\psi_i)+1$.
\begin{defn}[1-gap sequence]The \textbf{1-gap sequence at $P$} is $\{n_1,n_2,\dots,n_g\}$.
\end{defn}
We then have the following equivalent definition of a Weierstrass point.
\begin{defn}[Weierstrass point]
If the 1-gap sequence at $P$ is anything other than $\{1,2,\dots,g\}$, then $P$ is a Weierstrass point.
\end{defn}
With this formulation, we see $P$ is a Weierstrass point exactly when there is a holomorphic differential $f\cdot dx$ with $\ord_P(f\cdot dx)\geq g$.
\begin{defn}[Weierstrass weight]
The \textbf{Weierstrass weight} of a point $P$ is \[w(P)=\sum_{i=1}^g (n_i-i).\]
\end{defn}
In particular, $P$ is a Weierstrass point if and only if $w(P)>0$.

\subsubsection{Bounds for weights of Weierstrass points}
Suppose $\X_g$ is a curve of genus $g\geq1$, $P\in\X_g$, and consider the 1-gap sequence of $P$ $\{n_1,n_2,\dots,n_g\}$.  We will refer to the non-gap sequence of $P$ as the complement of this set within the set $\{1,2,\dots,2g\}$.  That is, the non-gap sequence is the sequence $\{\alpha_1,\dots,\alpha_g\}$ where \[1<\alpha_1<\dots<\alpha_g=2g.\]
\begin{prop}\label{prop:sums-of-non-gaps}
For each integer $j$ with $0<j<g$, $\alpha_j+\alpha_{g-j}\geq 2g$.
\end{prop}
\proof
Suppose there is some $j$ with $\alpha_j+\alpha_{g-j}<2g.$  The non-gaps are contained in a semigroup under addition, so for every $k\leq j$, since $\alpha_k+\alpha_{g-j}<2g$ as well, $\alpha_k+\alpha_{g-j}$ is also a non-gap which lies between $\alpha_{g-j}$ and $\alpha_g=2g$.  There are $j$ such non-gaps, though there can only be $j-1$ non-gaps between $\alpha_{g-j}$ and $\alpha_g$.  Thus, we have a contradiction.
\qed
\begin{prop}\label{prop:max-w-weight}
For $P\in\X_g$, $w(P)\leq g(g-1)/2$, with equality if and only if $P$ is a branch point on a hyperelliptic curve $\X_g$.
\end{prop}
\proof
The Weierstrass weight of $P$ is
\begin{align*}
w(P) & = \sum_{i=1}^g n_i- \sum_{i=1}^g i \\ 
&= \sum_{i=1}^{2g} i - \sum_{i=1}^g \alpha_i - \sum_{i=1}^g i \\
&= \sum_{i=g+1}^{2g-1} i - \sum_{i=1}^{g-1} \alpha_i.
\end{align*}
The first sum is $3g(g-1)/2$ and the second sum, via Proposition~\ref{prop:sums-of-non-gaps} is at least $(g-1)g$.  Hence, $w(P)\leq g(g-1)/2$.

To prove the second part, we note that the weight is maximized when the sum of the non-gaps is minimized.  That occurs when $\alpha_1=2$, which implies the non-gap sequence is $\{2,4,\dots,2g\}$, and so the 1-gap sequence is $\{1,3,5,\dots,2g-1\}$, which is the 1-gap sequence of a branch point on a hyperelliptic curve.
\qed

\begin{cor}\label{cor:number-w-pts}
For a curve of genus $g\geq2$, there are between $2g+2$ and $g^3-g$ Weierstrass points.  The lower bound of $2g+2$ occurs only in the hyperelliptic case.
\end{cor}
\proof The total weight of the Weierstrass points is $g^3-g$.  In Proposition~\ref{prop:max-w-weight}, we see that the maximum weight of a point is $g(g-1)/2$, which occurs in the hyperelliptic case.  Thus, there must be at least $\dfrac{g^3-g}{g(g-1)/2} = 2g+2$ Weierstrass points.  On the other hand, the minimum weight of a point is 1, so there are at most $g^3-g$ Weierstrass points.
\qed

\subsection{Higher-order Weierstrass points via holomorphic $q$-differentials}
In the above, we described Weierstrass points by considering the vector spaces $\mathcal{L}(K-nP)$ for $n\geq0$.  Now, we let $q\in\mathbb{N}$ and proceed analogously with the vector spaces $\mathcal{L}(qK-nP)$ to describe $q$-Weierstrass points.

If \[\ell(qK-(n-1)P)-\ell(qK-nP)=1,\] then there is some $q$-fold differential $dx^q$ and some $f\in k(\X_g)^\times$ such that $f\cdot dx^q$ is a holomorphic $q$-fold differential with $\ord_P(f\cdot dx^q)=n-1$.

Let $H^0(\X_g,(\Omega^1)^q)$ denote the space of holomorphic $q$-fold differentials on $\X_g$, and let $d_q$ denote the dimension of this space.  By Riemann-Roch, \[d_q=\begin{cases}g & \text{if $q=1$,} \\ (g-1)(2q-1) & \text{if $q>1$.}\end{cases}\]  Let $\{\psi_i\}$, for $i=1,\dots,d_q$, be a basis of $H^0(\X_g,(\Omega^1)^q)$, chosen in such a way that \[\ord_P(\psi_1)<\ord_P(\psi_2)<\cdots<\ord_P(\psi_{d_q}).\]  Let $n_i=\ord_P(\psi_i)+1$.
\begin{defn}[$q$-gap sequence]
The \textbf{$q$-gap sequence at $P$} is $\{n_1,n_2,\dots,n_{d_q}\}$.
\end{defn}
\begin{defn}[$q$-Weierstrass point]
If the $q$-gap sequence is anything other than $\{1,2,\dots,d_q\}$, then $P$ is a $q$-Weierstrass point.
\end{defn}
Thus, $P$ is a $q$-Weierstrass point exactly when there is a holomorphic $q$-fold differential $f\cdot dx^q$ such that $\ord_P(f\cdot dx^q)\geq d_q$.

When $q=1$, we have a Weierstrass point.  For $q>1$, a $q$-Weierstrass point is called a \textbf{higher-order Weierstrass point}.  
\begin{defn}[$q$-Weierstrass weight]
The \textbf{$q$-Weierstrass weight} of a point $P$ is \[w^{(q)}(P)=\sum_{i=1}^{d_q}(n_i-i).\]
\end{defn}
In particular, $P$ is a $q$-Weierstrass point if and only if $w^{(q)}(P)>0$.

\begin{rem}
For each $q\geq1$, there are a finite number of $q$-Weierstrass points.  This follows from Corollary~\ref{cor:total-q-weight} which says that the total $q$-weight of the $q$-Weierstrass points is finite.
\end{rem}
%***********************************************

\def\s{\sigma}
\def\fix{\mbox{Fix }}

%\noindent \textbf{Part  : Group actions on curves}
%**************************************************************

%\section{Group actions on Riemann surfaces}
%pg. 75

%\subsection{Quotient spaces}

%\subsection{Ramification of the quotient map}

\section{Hurwitz's theorem} 
In this section we will use the results of previous sections to study the automorphisms of algebraic curves. The main goal is to provide a proof of the Hurwitz's theorem on the bound of the order of the automorphism group. For any $\s \in \Aut (\X_g)$, we denote by $|\s|$ its order and $\fix (\s)$ the set of fixed points of $\s$ on $\X_g$.

\begin{prop}\label{prop:number-fixed-pts} Let $\s \in \Aut (\X_g)$ be a non-identity element. Then $\s$ has at most $2g+2$ fixed points.
\end{prop}
\proof 
Let $\s$ be a non-trivial automorphism of $\X_g$ and let $\s^*$ denote the corresponding automorphism of $k(\X_g)$.  Since $\s$ is not the identity, there is some $P\in\X_g$ not fixed by $\s$.  By Riemann-Roch, $\ell((g+1)P)\geq2$, so there is a meromorphic $f\in k(\X_g)$ with $(f)_\infty = rP$ for some $r$ with $1\leq r\leq g+1$.

Consider the function $h=f-\s^*(f)$.  The poles of $h$ are limited to the poles of $f$ and $\s^*(f)$, so $h$ has at most $2r$ poles.  Since $h$ is meromorphic, $h$ similarly has at most $2r$ zeroes, which correspond exactly to fixed points of $\s$.  Since $r\leq g+1$, we conclude $\s$ has at most $2g+2$ fixed points.
\qed

\begin{prop}
Any genus $g\geq2$ nonhyperelliptic Riemann surface $\X_g$ has a finite automorphism group $\Aut(\X_g)$.
\end{prop}

\proof
Let $\s\in\Aut(\X_g)$ with corresponding automorphism $\s^*$ of $k(\X_g)$.  The Wronskian does not depend on choice of local coordinate and thus is invariant under $\s^*$.  Therefore, if $P$ is a $q$-Weierstrass point of a certain $q$-Weierstrass weight, then $\s(P)$ is a $q$-Weierstrass point with the same weight.

Thus, any automorphism permutes the set of Weierstrass points.  Let $S_{WP}$ denote the permutation group of the set of Weierstrass points.  Since there are finitely many Weierstrass points (as in Corollary~\ref{cor:number-w-pts}), $S_{WP}$ is a finite group.  We have a homomorphism $\phi:\Aut(\X_g)\to S_{WP}$.  It will suffice to show that $\phi$ is injective.  We prove this separately in the cases that $\X_g$ is hyperelliptic or nonhyperelliptic.

Suppose $\X_g$ is non-hyperelliptic and suppose $\s\in\ker(\phi)$.  Then $\s$ fixes all of the Weierstrass points.  From Corollary~\ref{cor:number-w-pts}, since $\X_g$ is non-hyperelliptic, there are more than $2g+2$ Weierstrass points.  By Proposition~\ref{prop:number-fixed-pts}, $\s$ fixes more than $2g+2$ Weierstrass points and so must be the identity automorphism on $\X_g$. Thus, $\phi$ is an injection into a finite group, so $\Aut(\X_g)$ is finite.

Suppose $\X_g$ is hyperelliptic, and let $\omega\in\Aut(\X_g)$ denote the hyperelliptic involution.  Suppose $\s\in\ker(\phi)$ with $\s\neq\omega$.  $\s$ fixes the $2g+2$ branch points of $\X_g$.  Consider the map $\pi:\X_g\to\X_g/\langle\omega\rangle\cong\P^1$.  $\s$ descends to an automorphism of $\P^1$ which fixes at least $2g+2=6$ points, and so is the identity on $\P^1$.  Thus, $\s\in\langle\omega\rangle$, so $\s$ is the identity in $\Aut(\X_g)$, which means $\ker(\phi)$ is finite, so $\Aut(\X_g)$ is finite.
\qed

\begin{thm}[Hurwitz]
Any genus $g\geq 2$ Riemann surface $\X_g$ has at most $84 (g-1)$ automorphisms. 
\end{thm}

\proof 
Let $k(\X_g)$ denote the function field of $\X_g$.  From the above proposition, we know $\Aut(\X_g)$ is finite.  Let $n=|\Aut(\X_g)|$.  We wish to show $n\leq 84(g-1)$.

Let $L$ denote the fixed field of $k(\X_g)$ under the action of $\Aut(\X_g)$.  Then $L\subseteq k(\X_g)$ is a function field extension which corresponds to a morphism of curves $f:\X_g\to Y$.  Since $\Aut(\X_g)$ is finite, $f$ is a degree $n$ morphism.

Suppose $P\in\X_g$ is a ramification point of $f$ with ramification index $e_P=r$.  Let $f(P)=Q\in Y$.  Then $f^{-1}(Q)$ contains $n/r$ points, each with ramification index $r$.  %need to expand on this claim - need to expand on the ramification index earlier in this paper.

Let $Q_1,\dots,Q_s\in Y$ be the images of the ramified points of $X$.  For each $Q_i$, let $f^{-1}(Q_i)=\{P_{i,1},\dots,P_{i,k_i}\}.$  These points all have the same ramification index $r_i=e_{P_{i,j}}=n/k_i$, for all $j$.  By Riemann-Hurwitz, \[ 2g-2 = (2g(Y)-2)n+\sum_{i=1}^s\sum_{j=1}^{k_i}(n/k_i-1).\]  The right-hand side simplifies to give \[2g-2=(2g(Y)-2)n+n\sum_{i=1}^s(1-k_i/n),\] so \[(2g-2)/n=2g(Y)-2+\sum_{i=1}^s(1-1/r_i).\]

Since $g\geq2$, the left-hand side of this equation is positive, so the right-hand side must be positive as well.  We denote the right-hand side by $R$; that is, let \[R=2g(Y)-2+\sum_{i=1}^s(1-1/r_i)>0.\]  A minimal value of $R$ corresponds to a maximal value of $n$.  Thus, we aim to determine values of $g(Y)\geq0, s\geq0, r_1,\dots,r_s\geq2$ to minimize $R$.  To simplify things, we assume $r_1\leq r_2\leq\dots\leq r_s$.

If $s=0$, then $R=2g(Y)-2$, so $R\geq2$.

Now, suppose $s\geq1$.  If $g(Y)\geq1$, then $R\geq 2-2+(1-1/r_1)+\sum_{i=2}^s(1-r_i).$  Since the summation is strictly positive, this quantity is minimized when $s=1$ and $r_1=2$.  Hence, $R\geq1/2$.  Thus, $s\geq1$.

Now, if $g(Y)=0$, then $R=(s-2)-\sum_{i=1}^s 1/r_i.$  To get $R>0$, we need $s>2$.  Suppose $s=3$, and let $h(r_1,r_2,r_3)=1-1/r_1-1/r_2-1/r_3$.  Then $R=h(r_1,r_2,r_3)$.  If $r_1\geq4$, then the minimum value of $h(r_1,r_2,r_3)$, which occurs when $r_1=r_2=r_3=4$, is $1/4$.  If $r_1=3$, then the minimum (positive) value of $h(3,r_2,r_3)$, which occurs when $r_2=3$ and $r_3=4$, is $1/12$.

Now, suppose $r_1=2$. Then $h(2,r_2,r_3)=1/2-1/r_2-1/r_3$, so $r_2>2$.  Suppose $r_2=3$. Then $h(2,3,r_3)=1/6-1/r_3$, so $r_3=7$ gives a minimum value of $h(2,3,7)=1/42$.  If $r_2\geq4$, then $r_3\geq5$, and $h(2,r_2,r_3)\geq1/20$.  Thus, if $s=3$, then the minimum value of $R$ is $1/42$.

Now, we consider $s\geq4$ and $g(Y)=0$.  If $s=4$, then $R=2-1/r_1-1/r_2-1/r_3-1/r_4$, which, when $r_1=r_2=r_3=2, r_4=3$, has a minimum value of $1/6$.  If $s\geq5$, then $r_1=r_2=\dots=r_s=2$ gives $R\geq s/2-2\geq1/2$.

Having considered all possible combinations, we find the minimum value of $R$, which is $1/42$, occurs when $g(Y)=0$ and $(r_1,r_2,r_3)=(2,3,7)$.  Thus, \[(2g-2)/n\geq 1/42,\] so $n\leq84(g-1)$, as desired.

\qed

%\subsection{Fixed points of automorphisms}

The following two results consider the number of fixed points of an automorphism $\s \in \Aut (\X_g)$. %; see \cite{fk} for details. 

\begin{lem}
Let $\s\in\Aut(\X_g)$ be a non-trivial automorphism.  Then
\[ | \fix (\s)  \,  |  \, \leq 2 \, \frac {|\s| + g -1 } {|\s| -1}.\]  If $\X_g/\s\cong\P^1$ and $|\s|$ is prime, then this is an equality.
\end{lem}
\proof
Let $n=|\s|$, and for any $P\in\X_g$, let $O_P$ denote the orbit of $P$ under $\s$.  Then $|O_P|$ divides $n$.  Consider the degree $n$ cover $F:\X_g\to\X_g/\langle\s\rangle$ and let $Q\in\fix(\s)$.  Then $Q$ is ramified with $\mult_Q F=n$.  If $n$ is prime, then the fixed points are exactly the ramified points.  To see this, note that a non-fixed point $P$ has $|O_P|=n$ (because $n$ is prime) and so is unramified.

Now we apply Riemann-Hurwitz to this cover.  Let $g'$ denote the genus of $\X_g/\langle\s\rangle$.  Then $2g-2=n(2g'-2)+\sum_{P\in\X_g}(\mult_P F-1),$  so \[\sum_{Q\in\fix(\s)} (n-1) = 2g-2 -n(2g'-2)-\sum_{P\in\X_g\setminus\fix(\s)}(\mult_P F-1).\]  Thus $|(\fix(\s))|\leq \dfrac{2g-2+2n}{n-1}$, with equality when $g'=0$ and when the fixed points are exactly the ramified points.  That is, equality holds when $\X_g/\s\cong\P^1$ and $\s$ has prime order.
\qed
% \proof   % Prop on page 261 Farkas & Kra   \qed

\begin{cor} If $\X_g$ is not hyperelliptic, then for any non-trivial $\s \in \Aut (\X_g)$ the number of fixed points of $\s$ is $| \fix (\s) | \leq 2g-1$.
\end{cor}
\proof
If $\X_g$ is not hyperelliptic then $g\geq3$.  In the notation of the above proof, we have $g'\geq1$.  If $n=2$, then $|(\fix(\s))| = 2g+2-4g'$, so $|(\fix(\s))|\leq 2g-2$.  If $n\geq3$, then $|(\fix(\s))|\leq 2+\dfrac{2g-2g'n}{n-1}\leq 2+g$.

We can then combine these into one bound.  Note that $2g-2\leq 2g-1$ for all $g$, and $2+g\leq 2g-1$ for all $g\geq3$. Thus, if $\X_g$ is not hyperelliptic, then $|\fix(\s)|\leq 2g-1$.
\qed

\bigskip

\noindent \textbf{Part 3 : Weierstrass points on certain curves}
%******************************************

\section{Hyperelliptic and superelliptic curves} 
In this section, we give a brief background of hyperelliptic and superelliptic curves with some results related to Weierstrass points.  In particular, we describe how to calculate the $q$-Weierstrass weight of any branch point on a hyperelliptic or superelliptic curve.  Proofs can be found in \cites{towse-3, koo}.

\begin{defn}
A curve $X_g$, for $g\geq2$, is said to be {\it superelliptic} if there is a finite morphism $f:\X_g \to \P^1$ of degree $n$, for $n\geq2$.  A superelliptic curve is one which can be given in affine coordinates $x$ and $y$ by the equation $y^n=f(x)$, where $f(x)$ is a separable polynomial of degree $d>n$.

If $n=2$, then the curve is said to be hyperelliptic.
\end{defn}

Suppose $\X_g$ is given by $y^n=f(x)$ with $n\geq2$ and $f(x)\in\mathbb{C}[x]$ a separable polynomial of degree $d>n$.  Let $\{\alpha_1, \alpha_2, \dots, \alpha_d\}$ denote the $d$ distinct roots of $f(x)$, and for each $i$ let $B_i=(\alpha_i,0)$ be an affine branch point of the cover $\X_g\to\mathbb{P}^1$.  Let $c$ denote a complex number such that $f(c)\neq0$.  Let $P^c_1, \dots, P^c_n$ denote the $n$ points on $\X_g$ over $x=c$.

Let $G=\gcd(n,d)$.  All points on this model of the curve are smooth except possibly the point at infinity, which is singular when $d>n+1$.  In the smooth model of the curve, the point at infinity splits into $G$ points which we denote $P^\infty_1, \dots, P^\infty_G$.

One then has the following divisors:
\begin{itemize}
\item $(x-c) = \displaystyle\sum_{j=1}^nP^c_j - \dfrac{n}{G}\sum_{m=1}^GP^\infty_m,$
\item $(x-\alpha_i) = \displaystyle nB_i - \dfrac{n}{G}\sum_{m=1}^GP^\infty_m,$
\item $(y) = \displaystyle\sum_{j=1}^d B_j - \dfrac{d}{G}\sum_{m=1}^GP^\infty_m,$
\item $(dx) = (n-1)\displaystyle\sum_{j=1}^dB_j - \left(\dfrac{n}{G}+1\right)\sum_{m=1}^GP^\infty_m.$
\end{itemize}

Since $(dx)$ is a canonical divisor and hence has degree $2g-2$, we find the genus $g$ of $\X_g$ is given by \[2g-2=nd-n-d-\gcd(n,d).\]  In particular, if $n$ and $d$ are relatively prime, then $g=\dfrac{(n-1)(d-1)}{2}$.

Toward a basis for $H^0(C,(\Omega^1)^q)$, we first note that $\left(\displaystyle\frac{dx}{y^{n-1}}\right)=\dfrac{2g-2}{G}\displaystyle\sum_{m=1}^GP^\infty_m.$  Fix some $\alpha_i$ and $q\geq1$.  For any $a,b\in\Z$, let $h_{a,b,q}(x,y)=(x-\alpha_i)^ay^b\left(\frac{dx}{y^{n-1}}\right)^q.$  Then 
\[\left( h_{a,b,q}(x,y)\right) = anB_i+b\sum_{j=1}^dB_j+\frac{(2g-2)q-an-bd}{G}\sum_{m=1}^GP_m^\infty.\]  In particular, this divisor is effective precisely when $a\geq0$, $b\geq0$, and $an+bd\leq (2g-2)q.$  Since $y^n=f(x)$, the functions $h_{a,b,q}(x,y)$ are linearly independent if we let $a\geq0$ and restrict $b$ so that $0\leq b<n$.

Let \[S_{n,d,q}=\{(a,b)\in\Z^2 : a\geq0,\, 0\leq b<n,\, 0 \leq an+bd \leq (2g-2)q\}.\]  A counting argument gives the following lemma.

\begin{lem}
The set $S_{n,d,q}$ contains exactly $d_q$ elements.%, for $d_q=\begin{cases}g & \text{if } q=1, \\ (2q-1)(g-1) & \text{if } q>1.\end{cases}$
\end{lem}

From this set $S_{n,d,q}$, we create a basis $\mathfrak{B}_{q}=\{h_{a,b,q}(x,y) : (a,b)\in S_{n,d,q}\}$.  By the above lemma, since $\dim(H^0(C,(\Omega^1)^q))=d_q$, we have the following proposition.

\begin{prop}
For any root $\alpha_i$ and any $q\geq1$, the set $\mathfrak{B}_{q}$ forms a basis of $H^0(C,(\Omega^1)^q)$.
\end{prop}

One can then calculate the $q$-Weierstrass weight of any branch point $B_i=(\alpha_i,0)$ by calculating the orders of vanishing of the basis elements at $B_i$.  In particular, \[\ord_{B_i}\left(h_{a,b,q}(x,y)\right) = an+b.\]  Since $0\leq b<n$, these valuations are all distinct, and so \[w^{(q)}(B_i) = \sum_{(a,b)\in S_{n,d,q}} (an+b+1) - \sum_{m=1}^{d_q}m.\]

With this formula, we can show that any affine branch point is a $q$-Weierstrass point for all $q$.  First, we need a lemma.

\begin{lem}
For $\X_g$ a curve given by $y^n=f(x)$ with $f(x)$ separable of degree $d$, if $g>1$, then $g\geq n$ with equality only when $(n,d)=(2,5), (2,6),$ or $(3,4)$.
\end{lem}
\proof
One can check that if $(n,d)=(2,5), (2,6),$ or $(3,4)$, then $g=n$.

If $n=2$ and $d\geq7$, then $g=(d-\gcd(d,2))/2\geq3>n$.

If $n=3$ and $d\geq5$, then $g=(2d-1-\gcd(d,3))/2\geq4>n$.

If $n\geq4$, then $d\geq5$, and so $2g=(n-1)(d-1)-\gcd(n,d)+1\geq (n-1)(d-2) \geq 3(n-1)$.  Thus, $g\geq \frac{3}{2}(n-1)$, which is larger than $n$ for $n>3$.
\qed

\begin{prop}
Any affine branch point $B_i$ is a $q$-Weierstrass point for all $q\geq1$.
\end{prop}
\proof
We first consider the case where $q=1$.  The function $1/(x-\alpha_i)$ has a pole only at $B_i$ of order $n$, so $\ell(nB_i)>1$.  By the above lemma, $g\geq n$, so $\ell(gB_i)\geq \ell(nB_i) >1$, which implies $B_i$ is a $1$-Weierstrass point.

Now, suppose $q>1$.
Since there are $d_q$ distinct positive terms in both summations in \[w^{(q)}(B_i) = \sum_{(a,b)\in S_{n,d,q}} (an+b+1) - \sum_{m=1}^{d_q}m,\] the $q$-Weierstrass weight of $B_i$ will be positive precisely when there is some $(a,b)\in S_{n,d,q}$ such that $\ord_{B_i}(h_{a,b,q}(x,y))>d_q$.  Let $A= \left\lfloor \frac{(2g-2)q}{n} \right\rfloor$, where $\lfloor x\rfloor$ is the floor function.  Then \[An+0d=\left\lfloor\frac{(2g-2)q}{n}\right\rfloor n\leq(2g-2)q,\] which means $(A,0)\in S_{n,d,q}$.  The order of vanishing at $B_i$ is given by  \[ \ord_{B_i}(h_{A,0,q}(x,y)) = \left\lfloor \frac{(2g-2)q}{n} \right\rfloor n+1 \geq (2g-2)q-(n-1) +1,\] with the $(n-1)$ term representing the maximal fractional part of $(2g-2)q/n$ multiplied by $n$.

If $q>1$, then $\ord_{B_i}(h_{A,0,q}(x,y))-d_q \geq g-n+1$, which is at least 1 by the above lemma.  Thus, $\ord_{B_i}(h_{A,0,q}(x,y))>d_q$, so $B_i$ is a $q$-Weierstrass point.

\qed

%***************************************************
\section{The group action on the Weierstrass points of non-hyperelliptic curves of genus 3.}

Here we give the group action on the set of Weierstrass points on non-hyperelliptic curves of genus 3. The case of $g=2$ and $g=3$ hyperelliptic is trivial.  This summarizes the work done in \cite{babu}. 

A genus 3 curve $X_3$ is either hyperelliptic or it is a non-singular plane quartic. In the hyperelliptic case $X_3$ has exactly 8 Weierstrass points which are the ramification points of the canonical map $X_3 \rightarrow {\mathbb{P}}^1$. The action of $\Aut(X_3)$ on the Weierstrass points can be easily deduced from the previous section; see also   \cite[Tab.~3]{kyoto}.

Let $X_3$ be a non-singular plane quartic. Its Weierstrass points are the inflection points. The weight of a Weierstrass point is the same as the multiplicity of the inflection point which is either 1 or 2; see the end of Section~\ref{sec:linear-systems}.  The weighted number of inflection points is 24; see Theorem~\ref{thm:total-inf-weight}.

The inflection points are the intersection points of $X_3$ with its Hessian. The tangent to $X_3$ at the inflection point hits $X_3$ in $N$ further points, where $N=2$ (respectively $1$) if the weight of the inflection point is $1$ (respectively $2$). This way the inflection points and their weights can be computed effectively from the equation of the curve.

The groups $G$ occurring as $\Aut(X_3)$ are denoted in \cite{kyoto} by their group ID from the GAP library of small groups. Here we use the following shorter notation: $C_n$, $S_n$, $D_{2n}$, $V_4$ denotes the cyclic group of order $n$, the symmetric group on $n$ letters, the dihedral group of order $2n$, and the Klein-$4$ group respectively. Also, $L_{3}(2)$ ($=PGL_{3}(2)$) is the simple group of order 168. The groups of order $16$, $48$ and $96$ are just denoted by their group order. Their group IDs are $(16,13)$, $(48,33)$ and $(96,64)$ respectively. The group $(96,64)$ is sometimes denoted $C_4^2.S_3$.

We show how to derive the information on the Weierstrass points given in \cite[Tab.~1]{babu} and \cite[Tab.~2]{babu}.

When $G=L_{3}(2)$ or $C_{4}^{2}.S_{3}$, then $G$ has only one orbit of length $\leq 24$ on $X_3$. Therefore this orbit has to consist of all Weierstrass points. For $G=L_{3}(2)$ they form one orbit of length $24$ with a stabilizer of order $7$ and they all have weight $1$. For $G=C_{4}^{2}.S_{3}$ the Weierstrass points have weight $2$ and form one orbit of length $12$ with  stabilizer of order $8$.

From \cite[Tab.~3]{babu} we see that the $L_{3}(2)$-locus (a single point) is contained in each of the loci with groups $S_{4}$, $D_{8}$, $S_{3}$, $V_{4}$, $C_{2}$. Therefore the general $X_{3}$ in each of these loci has $24$ distinct Weierstrass points of weight one. No element of order $2$ or $3$ in $L_{3}(2)$ fixes a Weierstrass point of the corresponding $X_{3}$, therefore the same holds for the general curve in the loci listed above. Therefore for all $X_{3}$ in these loci, the Weierstrass points have weight $1$ and consist of regular orbits. This settles the
corresponding entries of \cite[Tab.~1]{babu}.

When $G=(48,33)$, we have three non-regular $G$-orbits of length $24$, $16$, and $4$. Let $S$ denote a Sylow $3$-subgroup of $G$ (of order $3$). Then $S$ stabilizes a point in the $4$-orbit as well as in the $16$-orbit (because these orbits have cyclic stabilizers of order $12$ and $3$ respectively). Let $N(S)$ denote the normalizer of $S$ in $G$. Then the number of fixed points of $S$ in the $16$-orbit is $\mid \frac{N(S)}{S} \mid $. Since $N(S)$ contains the cyclic group of order $12$ (point stabilizer of the $4$-orbit) it follows that $\mid \frac{N(S)}{S} \mid \geq4$. Thus $S$ has at least 4 fixed points in the $16$-orbit, and it fixes at least one point in the $4$-orbit. Thus $S$ has at least $5$ fixed points. By \cite[Theorem V.1.7]{fk} all fixed points of $S$ are Weierstrass points. It follows all points in the $16$-orbit and the $4$-orbit are Weierstrass points. Thus the points in the $16$-orbit have weight $1$ and the points in the $4$-orbit have weight $2$.

When $G=(16,13)$, there are four non-regular $G$-orbits of length  $8$, $8$, $8$, $4$. The $4$-orbit consists of the points $(0,0,1)$, $(1,0,1)$, $(1,0,0)$, $(t,0,1)$. We check that the Hessian of the equation $y^4=xz(x-z)(x-tz)$ is $0$ at $(0,0,1)$ so this point is a Weierstrass point. It follows all points in the $4$-orbit are Weierstrass points. Further, the $16$-locus contains the $48$-locus. Therefore by specialization, the general curve in the $16$-locus has one regular orbit of Weierstrass points with weight $1$. Thus the Weierstrass points of the $16$-locus consist of one regular orbit with points of weight $1$ and one $4$-orbit with points of weight $2$.

When $G=C_{3}$, $C_{6}$, $C_{9}$, the subgroup $C_{3}$ fixes five points: $(0,0,1)$, $(1,0,1)$, $(s,0,1)$, $(t,0,1)$, $(0,1,0)$. By \cite[Theorem V.1.7]{fk} they are all Weierstrass points. Using  
\cite[Remark 1]{babu} we compute the weight of these points, and get that the points $(0,0,1)$, $(1,0,1)$, $(s,0,1)$, $(t,0,1)$ have weight $1$, and $(0,1,0)$ has weight $2$.

Furthermore, we compute  the Hessian $h(x,y,z)$ of the $C_9$-equation $f(x,y,z)$. Since the Weierstrass points are the intersection of the Hessian with $X_3$, we set $z=1$ and consider the system of
equations $h(x,y,1)=0$ and $f(x,y,1)=0$. The resultant with respect to $x$ of these two polynomials is a polynomial in $y$ that has degree 18 and has nonzero discriminant. Thus the $X_3$ in the $C_9$-locus has $18$ distinct Weierstrass points in addition to the 5 fixed points of the $C_3$-subgroup. Thus the Weierstrass points consist of $2$ regular orbits with points of weight $1$, plus the above $5$ fixed points of $C_3$.

Now consider $X_3$ in the $C_3$-locus. By specialization to the $C_9$-locus we see that $X_3$ has $22$ Weierstrass points of weight $1$. This settles the $C_3$ case. 
The $48$-group occurs also in \cite[Tab.~1]{babu}, and so has already been dealt with.

Finally, we consider $X_3$ in the $C_6$-locus. The non-regular orbits consist of one fixed point and three additional orbits of length 3, 2, 2 with stabilizing subgroups of order 2, 3, 3 respectively. Now $C_6$ has only one subgroup of order $3$ so $C_3$ fixes all points in the 2-orbits. Thus we obtain the 5 fixed (Weierstrass) points of $C_3$ whose weights we computed above.

By specialization to the 48-locus we see that $X_3$ has at least 16 distinct Weierstrass points of weight 1. Thus in addition to the 5 fixed points of $C_3$, the Weierstrass points must consist of either 3 regular orbits with points of weight 1, or 2 regular orbits with points of weight 1 and the 3-orbit with points of weight 2. To determine which, we compute the fixed points on $X_3$ of the involution $N$ in $C_6$. They are the points $(1,y_i,2)$, $1 \leq i \leq 3$, where the $y_i$ are the roots of the equation $y^3=\frac{(1-2t)^2}{2}$. We find that for $t \neq \frac{1}{2}$, $\frac{1+ \zeta_{4}}{2}$, $\frac{1- \zeta_{4}}{2}$, the Hessian of $X_3$ is not 0 at these points. Thus for these values of $t$ the Weierstrass points are a union of 3 regular orbits with points of weight 1, and the five fixed points of $C_3$.

%************************************************************
% \noindent \textbf{Part  : Computing $q$-Weierstrass points}

%*********************************************************************************************************
\nocite{*}  %comment this if you want only references that are being used to show up

\bibliographystyle{amsplain} 

\bibliography{w-bibl}{}

%*********************************************************************************************************

\end{document}